\documentclass[12pt]{amsart}
\usepackage{amsfonts, amsmath, amssymb}
\usepackage{amsthm}
\usepackage{graphicx}
\usepackage{float}
\usepackage{enumitem}
\usepackage{color}

\usepackage[utf8]{inputenc} 
\usepackage{enumitem}
\usepackage{lipsum}  

\usepackage{graphicx}
\usepackage{subcaption}
\usepackage[utf8]{inputenc}
\usepackage{tikz}
\usetikzlibrary{calc}

\newcommand{\tikzmark}[1]{\tikz[overlay,remember picture] \node (#1) {};}
\newcommand{\DrawBox}[1][]{%
    \tikz[overlay,remember picture]{
    \draw[red,#1]
      ($(left)+(-0.2em,0.9em)$) rectangle
      ($(right)+(0.2em,-0.3em)$);}
}

\newtheorem{Theorem}{Theorem}
\numberwithin{Theorem}{section}

\newtheorem{Proposition}[Theorem]{Proposition}

\theoremstyle{definition}

\theoremstyle{remark}

\newtheorem{Example}[Theorem]{Example}

\expandafter\chardef\csname pre amssym.def
at\endcsname=\the\catcode`\@ \catcode`\@=11
\def\undefine#1{\let#1\undefined}
\def\newsymbol#1#2#3#4#5{\let\next@\relax
 \ifnum#2=\@ne\let\next@\msafam@\else
 \ifnum#2=\tw@\let\next@\msbfam@\fi\fi
 \mathchardef#1="#3\next@#4#5}
\def\mathhexbox@#1#2#3{\relax
 \ifmmode\mathpalette{}{\m@th\mathchar"#1#2#3}%
 \else\leavevmode\hbox{$\m@th\mathchar"#1#2#3$}\fi}
\def\hexnumber@#1{\ifcase#1 0\or 1\or 2\or 3\or 4\or 5\or 6\or 7\or 8\or
 9\or A\or B\or C\or D\or E\or F\fi}

\font\teneufm=eufm10 \font\seveneufm=eufm7 \font\fiveeufm=eufm5
\newfam\eufmfam
\textfont\eufmfam=\teneufm \scriptfont\eufmfam=\seveneufm
\scriptscriptfont\eufmfam=\fiveeufm

\catcode`\@=\csname pre amssym.def at\endcsname

\begin{document}





\title[]{The number of perpendicularly inscribed polygons that intersect a given side in an odd sided regular polygon}

\author[J. A. M. Gondim]{João A. M. Gondim}


\bigskip

\maketitle

\centerline {Unidade Acadêmica do Cabo de Santo Agostinho} \par \centerline{
Universidade Federal Rural de Pernambuco}

\begin{abstract}
 The goal of this paper is to determine the number of perpendicularly inscribed polygons that intersect a given side of a regular polygon with an odd number of sides. This is done using circular permutations with repetition, and some special cases are calculated via circulant matrices and the Binomial Theorem. A method for finding such polygons, based on Banach's Fixed Point Theorem, is also developed.
\end{abstract}

\vspace{0.4cm}
\centerline{ {\bf Key Words:} 
Inscribed Polygons, Circular Permutations, Circulant Matrices} \vspace{0.4cm} 

\section{Introduction}

A polygonal chain $P$ defined by the sequence $(A_1, A_2, \ldots, A_k)$ is \textit{inscribed} in a simple polygon $Q$ if every vertex of $P$ lies in a side of $Q$. Moreover, we will say that $P$ is \textit{perpendicularly inscribed} in $Q$ if the line segment $A_iA_{i+1}$ is perpendicular to the side of $Q$ in which $A_i$ lies. We also refer to $P$ as an \textit{orbit}.

This paper focuses in the case in which $P$ closes after $n$ steps (i.e., $A_{n+1} = A_1$), so $P$ is a polygon perpendicularly inscribed in the simple polygon $Q$. As before, we also refer to such a polygon as a \textit{periodic orbit} with period $n$, or a $n$-periodic. See Figure \ref{Fig-1}.

\begin{figure}
\begin{minipage}[b]{.5\linewidth}
\centering
\includegraphics[scale=0.8,trim={8cm 0cm 9cm 0cm}]{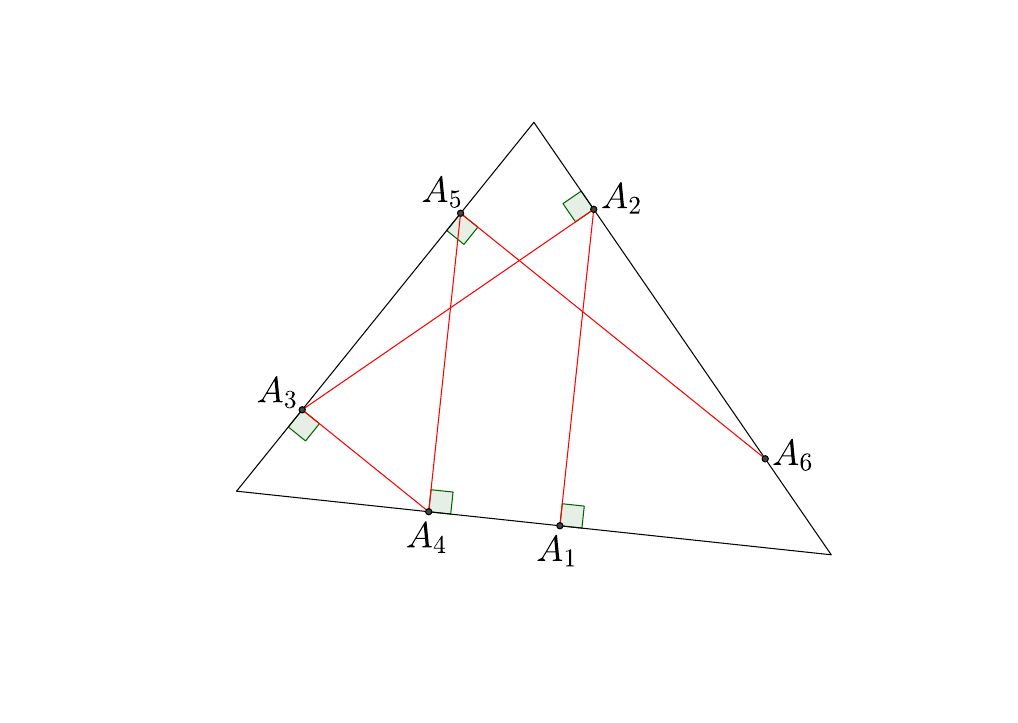}
\end{minipage}%
\begin{minipage}[b]{.5\linewidth}
\centering
\includegraphics[scale=0.8,trim={8cm 0cm 9cm 0cm}]{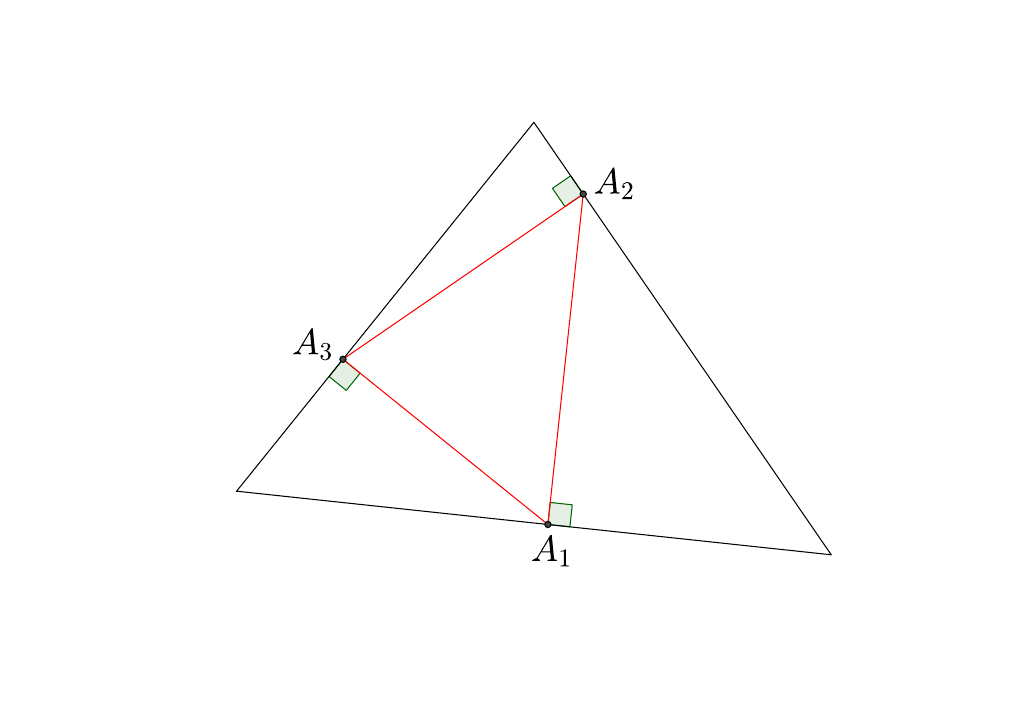}
\end{minipage}
\caption{On the left, a polygonal chain with 6 points perpendicularly inscribed in a triangle. On the right, a triangle perpendicularly inscribed in another triangle.}\label{Fig-1}
\end{figure}

The problem of finding a polygon inscribed in some kind of geometric object has been studied in the context of billiards, for example. In elliptic billiards \cite{chang1988elliptical,reznik2020can,reznik2020ballet,bialy2020dan,akopyan2020billiards}, this is related to Poncelet's Porism \cite{levi2007poncelet}, whereas for polygonal billiards this is a harder problem \cite{biswas1997periodic}. For triangular billiards, the problem has been solved in some special cases \cite{baxter2008periodic,cipra1995periodic}. Non-billiard papers, which are more closely related to this one, appear in the works of Lowry \cite{lowry1952polygons} and Pamfilos \cite{pamfilos2011polygons}.

The goal of this paper is to determine the number of perpendicularly inscribed polygons that intersect a given side of an odd sided regular polygon. Regular polygons with an even number of sides are not considered since the periodic orbits produced will always have period $2$.

In Section 2, an argument based on Banach's Fixed Point Theorem that allows us to find the periodic orbits is presented. Section 3 develops the main result, counting the periodic orbits using circular permutations with repetition \cite{macmahon1891applications}. Some special cases can be easily calculated using circulant matrices \cite{feng2013note}. This is done in Section 4. Section 5 concludes the work with applications of the formula deduced in Section 3.

\section{A method to produce periodic orbits}

Consider an odd sided regular polygon $Q$. Its sides are labeled $1$, $2$, $\ldots$, $2k+1$ counterclockwise as in Figure \ref{labels}. This is done simply to distinguish the orbits by the sequence of sides that are visited by the polygonal chain, starting at side $1$.

\begin{figure}[h!]
    \centering
    \includegraphics[scale=0.9,trim={5.5cm 1.5cm 5.5cm 0.5cm}]{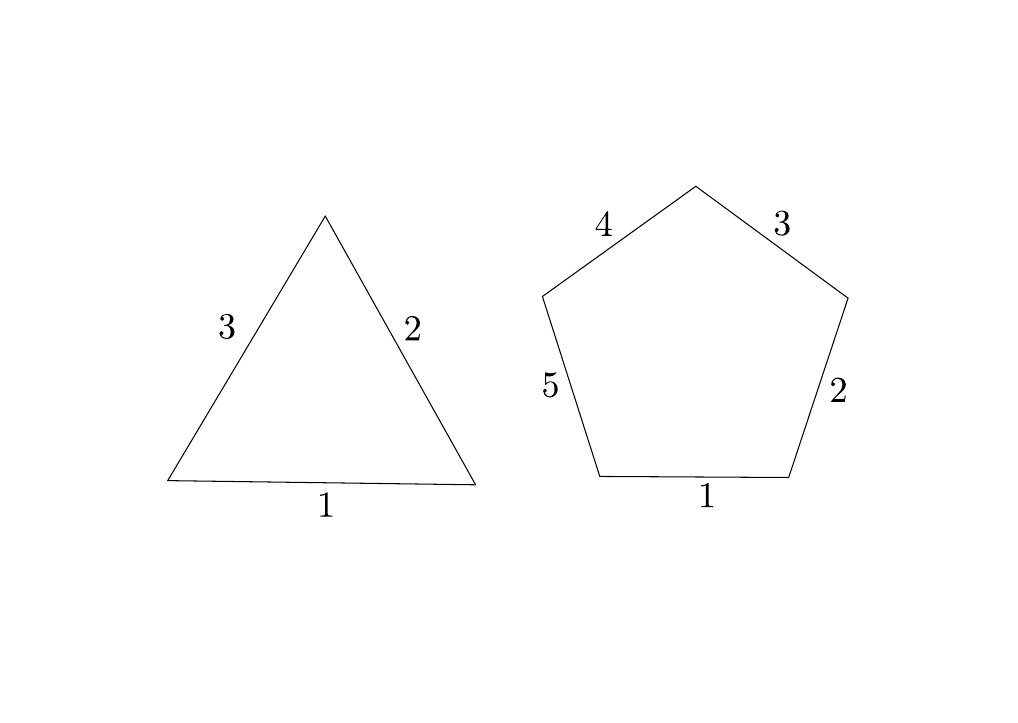}
    \caption{The labels of the regular polygon's sides.}
    \label{labels}
\end{figure}

Choose two points $P_1$ and $Q_1$ on a given side $i$ of $Q$ such that the lines perpendicular to $i$ through $P_1$ and $Q_1$ intersect the triangle on points $P_2$ and $Q_2$, respectively, that also lie on the same side. Let $\theta$ be the measure of the internal angles of $Q$, and let $x_1$ and $x_2$ be the measures of the segments $P_1Q_1$ and $P_2Q_2$, respectively. Then, as Figure \ref{contrac-fig-1} indicates,
\begin{equation}
    x_2 = \frac{x_1}{\sin(\theta/2)}.
    \label{contrac-1}
\end{equation}

\begin{figure}[h!]
    \centering
    \includegraphics[scale=1,trim={5.5cm 1cm 5.5cm 0.5cm}]{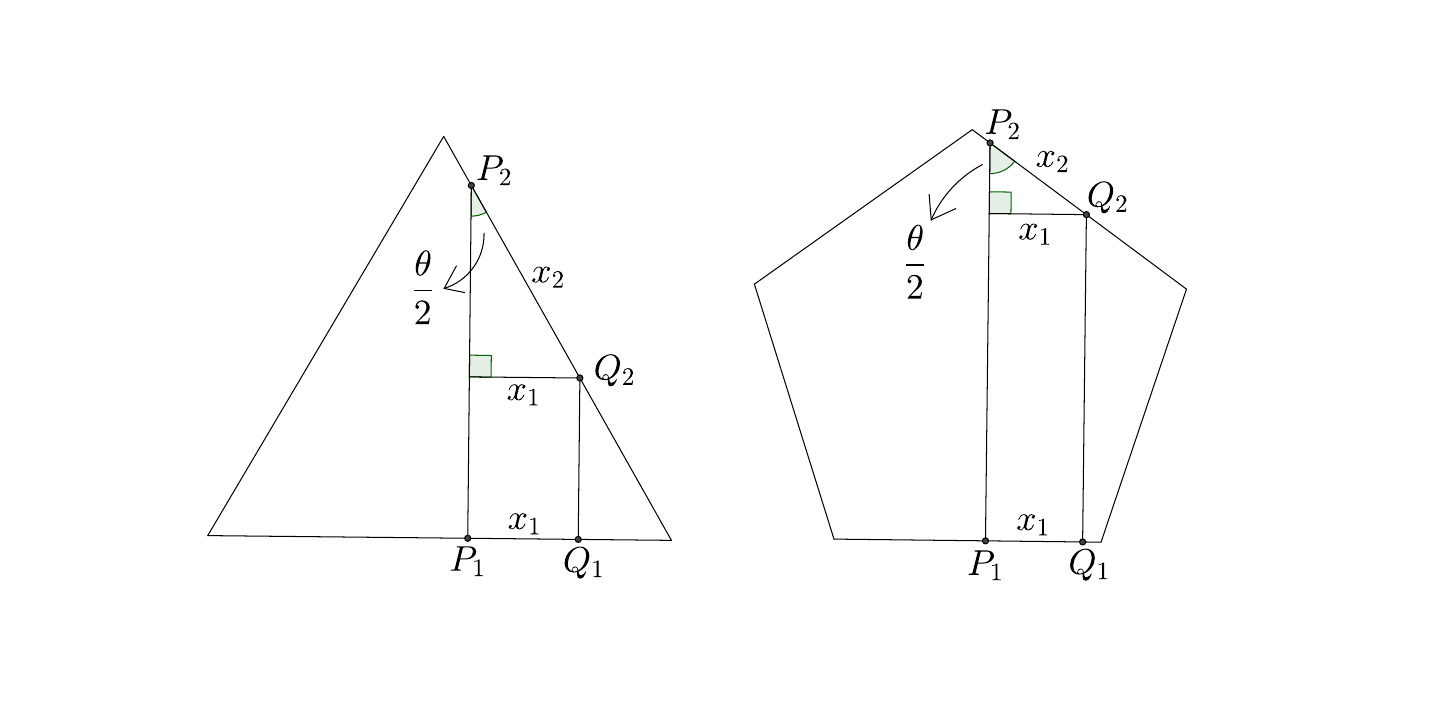}
    \caption{Figure for \eqref{contrac-1}.}
    \label{contrac-fig-1}
\end{figure}

Repeating this argument, if the orbits of $P_1$ and $Q_1$ return to their original side after $n$ iterations, then the segment defined by the returning points $P_n$ and $Q_n$ has length $x_n$ satisfying
\begin{equation}
    x_n = \frac{x_1}{\sin^n(\theta/2)}.
    \label{contrac-2}
\end{equation}

Hence, the backward orbit map with domain $P_nQ_n$ and side $i$ as codomain is a contraction. The backward orbit of any point between $P_n$ and $Q_n$ converges to an unique fixed point by Banach's Fixed Point Theorem. If the chosen sequence of sides visits only two of $Q$'s sides, then this fixed point is a vertex of $Q$ (equivalently, there are no $2$-periodics in odd sided regular polygons). If not, then it is a vertex of a periodic orbit with period $n$.

\section{The number of perpendicularly inscribed polygons that visit a given side}

Let $Q$ be a regular polygon with $2k+1$ sides. We wish to count all periodic orbits with period $n$ that pass through side $1$. Let $O_{2k+1}(n)$ denote this number. Without loss of generality, we can consider side $1$ to be the first side visited by the orbit and let $$\alpha = (1,\alpha_2,\ldots, \alpha_n)$$ be the sequence of sides it visits before returning to side $1$. Also, suppose that more than two different sides appear in $\alpha$.

Notice that the lines perpendicular to side $i$ can only intersect either side $i+k$ or side $i+k+1$, where additions are performed modulo $2k+1$. Consider a graph $G_{2k+1}$ with $2k+1$ vertices such that on vertex $i$ there are only edges adjoining it to vertices $i+k$ and $i+k+1$ modulo $2k+1$ and let $A_{2k+1}$ be its adjacency matrix. The entry $(1,1)$ in $A_{2k+1}^n$ counts the number of sequences of $n$ sides that start in side $1$ and also end in side $1$.

Besides the sequences that involve only two of $Q$'s sides, there are other sequences that prevent this entry from being our desired number. For example, if $\alpha$ is such that $\alpha_i = 1$ for some $i \in \{3, \ldots, n-1\}$, then $\alpha$ is counted more than once, i.e., periodic orbits produced from sequences such as $123132$ and $132123$ are the same, only the starting point on side $1$ is changed. The same happens whenever a sequence is a cyclic permutation of another.

Hence, we break $\alpha$ into smaller pieces, each with $1$ appearing only as the first element. We call these \textit{pure orbits} (of length $r$). The number of pure orbits of length $r$ in $Q$, which will be denoted by $P_{2k+1}(r)$, can be calculated as follows. Consider a graph $H_{2k+1}$ obtained from $G_{2k+1}$ by deleting both edges that incide on vertex $1$ and let $B_{2k+1}$ be its adjacency matrix. $P_{2k+1}(r)$ is given by the sum of entries of the submatrix of $B_{2k+1}^{r-2}$ given by rows and columns in the set $\{k+1,k+2\}$. Notice that we are considering $P_{2k+1}(2) = O_{2k+1}(2)= 2$ for every $k$.

Therefore, we have to consider partitions of $n$ in parts that are at least $2$, since it is necessary to reach at least one other side before returning to side $1$. Every part $r$ in the partition represents a pure orbit of length $r$. We first consider partitions with at least two distinct parts. 

Let $p$ be a partition of $n$ into parts $p_1, p_2, \ldots, p_m$ that are all greater than $2$. The number $F(p)$ of partitions that are not cyclical permutations of $p$ corresponds to the number of circular permutations of $m$ parts out of $r$ different kinds, with $\beta_1$ of the first kind, $\beta_2$ of the second and so forth. Let $CP(\beta_1, \ldots, \beta_r)$ and $LP(\beta_1, \ldots, \beta_r)$ be the numbers of circular and linear permutations of these $k$ objects. By \cite{macmahon1891applications}, $F(p)$ is given by
\begin{equation}
    CP(\beta_1, \ldots, \beta_r) = \dfrac{1}{m} \sum_{d \mid m} \phi(d) LP\left(\frac{N}{d}\beta'_1, \ldots, \frac{N}{d}\beta'_r\right)
    \label{circperm-1}
\end{equation}
where $N = \gcd(\beta_1, \ldots, \beta_r)$, $\beta'_i = \beta_i/N$ and $\phi$ is Euler's totient function.

Moreover, there are $P_{2k+1}(p_i)$ pure orbits, for $i \in \{1, \ldots, m\}$. Hence, the number of sequences of $n$ sides induced by partitions of $n$ that have at least two different parts is
\begin{equation}
    \sum_{p \in S_1(n)} F(p) \prod_{i=1}^m P_{2k+1}(p_i),
    \label{formula-par1}
\end{equation}
where $S_1(n)$ is the set of partitions of $n$ with parts greater than $1$ and, at least, two distinct ones among them.

On the other hand, partitions of $n$ with equal parts consist of a divisor $d \neq 1$ of $n$ added $n/d$ times. For every part, we must choose one pure $d$-periodic among the $\ell_d = P_{2k+1}(d)$ possibilities. Hence, if $\beta_1$, $\ldots$, $\beta_{\ell_d}$ are the numbers of pure $d$-periodics of each kind, then 
\begin{equation}
    \sum_{i=1}^{\ell_d} \beta_i = \dfrac{n}{d}.
    \label{soma}
\end{equation}

Consequently, if $S_2(n/d)$ is the set of nonnegative integer solutions of $\eqref{soma}$ and $x \in S_2(n/d)$, then the number of sequences of $n$ sides induced by these partitions is
\begin{equation}
    \sum_{\substack{d \mid n \\ d \neq 1}} \sum_{x \in S_2(n/d)} CP(\hat{x}),
    \label{formula-2}
\end{equation}
where $\hat{x}$ denotes the nonzero terms in $x$.

Finally, notice that for every divisor $d \neq 1$ of $n$, the $d$-periodics traveled $n/d$ times are counted above, so we need to subtract 
\begin{equation}
    \sum_{\substack{d \mid n \\ d \neq 1}} O_{2k+1}(d).
    \label{formula-3}
\end{equation}

Notice that when we subtract $O_{2k+1}(2) = 2$ we are disregarding the sequences of sides that visit only two of them. Combining \eqref{formula-par1}, \eqref{formula-2} and \eqref{formula-3} we get 
\begin{equation}
    \begin{aligned}
        O_{2k+1}(n) &= \sum_{p \in S_1(n)} F(p) \prod_{i=1}^m P_{2k+1}(p_i)  \\
       &+ \sum_{\substack{d \mid n \\ d \neq 1}} \left(\sum_{x \in S_2(n/d)} CP(\hat{x}) -  O_{2k+1}(d)\right).
    \end{aligned}
    \label{model1}
\end{equation}

Some examples of the formula above will be discussed in  Section 5.

\section{Results regarding some special cases}

The $i$-th line of the adjacency matrix $A_{2k+1}$, defined in the previous Section, has nonzero elements (which are ones) in columns $i+k$ and $i+k+1$ modulo $2k+1$. It follows that $A_{2k+1}$ is a circulant matrix. To easily compute powers of a circunlant matrix, we follow \cite{feng2013note} and consider the matrix $C_{2k+1}$ such that its $i$-th line consists entirely of zeros, except for a $1$ in column $i+1$.

$A_{2k+1}$ can be written, then, as 
\begin{equation}
    A_{2k+1} = C_{2k+1}^k + C_{2k+1}^{k+1}.
    \label{circulant}
\end{equation}

Notice, also, that $C_{2k+1}^{2k+1} = I_{2k+1}$, the identity matrix. We now prove some interesting results.

\begin{Proposition}
$O_{2k+1}(2k+1) = 2$ for every $k \geq 1$.
\label{prop-1}
\end{Proposition}

\begin{proof}
The paths $1 \rightarrow k+1 \rightarrow 2k+1 \rightarrow \ldots \rightarrow (2k+1)k+1$ and $1 \rightarrow (k+1)+1 \rightarrow 2(k+1)+1 \rightarrow \ldots \rightarrow (2k+1)(k+1)+1$ (modulo $2k+1)$ are $(2k+1)$-periodics. We show that these are the only ones. In order to calculate the $(1,1)$ entry of $A_{2k+1}^{2k+1}$, we use the Binomial Theorem and write
\begin{align*}
    A_{2k+1}^{2k+1} &= \left(C_{2k+1}^k + C_{2k+1}^{k+1}\right)^{2k+1} \\
    &= \sum_{j=0}^{2k+1} \binom{2k+1}{j}C_{2k+1}^{kj}C_{2k+1}^{(k+1)(2k+1-j)}.
\end{align*}

The exponent of $C_{2k+1}$ in the last expression can be simplified to $$(k+1)(2k+1)-j.$$ Since we are looking for the $(1,1)$ entry (actually all entries in the diagonal are the same), we want the coefficient of $C_{2k+1}^{2k+1} = I_{2k+1}$ in this expansion. The exponent is a multiple of $2k+1$ if $j=0$ or if $j = 2k+1$, so the desired coefficient is $$\binom{2k+1}{0}+\binom{2k+1}{2k+1} = 2.$$ Hence, there are only two ways to start in side $1$ and return to it after $2k+1$ steps. Illustrations for $O_3(3)$ and $O_5(5)$ are provided in Figure \ref{orbitas-prop4-1}.

\end{proof}

\begin{figure}[h!]
    \centering
    \includegraphics[scale=0.7,trim={5.5cm 1cm 5.5cm 0.5cm}]{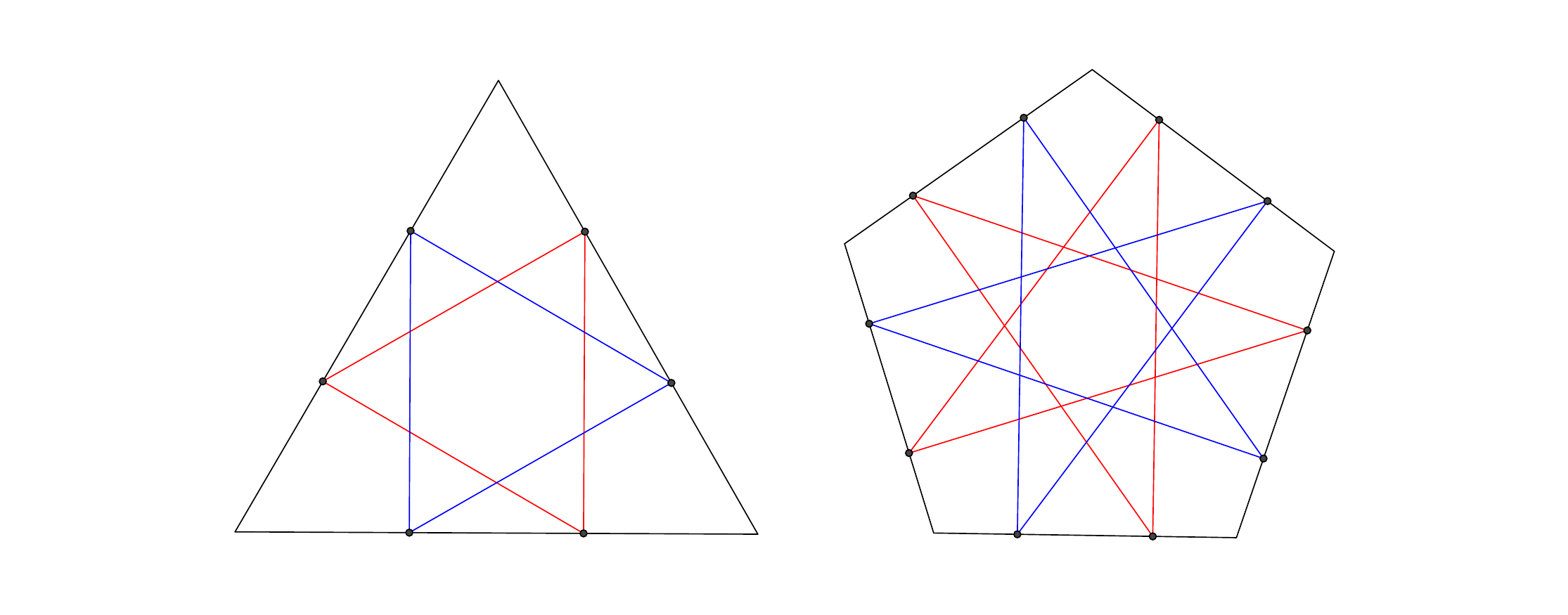}
    \caption{Illustrations for $O_{2k+1}(2k+1) = 2$ for $k = 1$ (triangle) and $k = 2$ (pentagon).}
    \label{orbitas-prop4-1}
\end{figure}

\begin{Proposition}
$O_{2k+1}(2t+1) = 0$ if $1 \leq t < k$.
\label{prop-2}
\end{Proposition}

\begin{proof}
We now have 
\begin{align*}
    A_{2k+1}^{2t+1} &= \left(C_{2k+1}^k + C_{2k+1}^{k+1}\right)^{2t+1} \\
    &= \sum_{j=0}^{2t+1} \binom{2t+1}{j} C_{2k+1}^{kj} C_{2k+1}^{(k+1)(2t+1-j)}.
\end{align*}

Now, a simplification of the exponent yields $$(2t+1)(k+1)-j.$$ 

Since $j \in \{0,1, \ldots, 2k+1\}$, we have $$(2t+1)k = (2t+1)(k+1)-(2t+1) \leq (2t+1)(k+1)-j \leq (2t+1)(k+1),$$ so the exponent cannot equal any multiple of $2k+1$, hence the coefficient of the identity matrix in the expansion is zero. In particular, there is no $(2t+1)$-periodic in a regular polygon with $2k+1$ sides for $t < k$.

\end{proof}

\begin{Proposition}
$O_{2k+1}(4) = 3$ for all $k \geq 1$.
\label{prop-3}
\end{Proposition}

\begin{proof}
As before, we write
\begin{align*}
    A_{2k+1}^4 &= \left(C_{2k+1}^k + C_{2k+1}^{k+1}\right)^4 \\
    &= \sum_{j=0}^{2t+1} \binom{4}{j} C_{2k+1}^{kj} C_{2k+1}^{(k+1)(4-j)}.
\end{align*}

The exponent of $C_{2k+1}$ is $$4k+4-j,$$ which is a multiple of $2k+1$ only for $j = 2$. The coefficient in the binomial expansion is $\dbinom{4}{2} = 6$. These include the sequences $(1,k+1,1,k+1)$ and $(1,k+2,1,k+2)$, which do not produce periodic orbits, and the sequences $(1,k+1,1,k+2)$ and $(1,k+2,1,k+1)$, which produce the same periodic orbit. Hence, out of the $6$ walks that arise from the $(1,1)$ entry of $A_{2k+1}^4$, only $3$ remain: $(1,k+1,2k+1,k+1)$, $(1,k+2,2,k+2)$ and $(1,k+1,1,k+2)$. Figure \ref{Fig-Prop4-3} illustrates the cases of $O_3(4)$ and $O_5(4)$.

\end{proof}

\begin{figure}[h!]
    \centering
    \includegraphics[scale=0.7,trim={5.5cm 1cm 5.5cm 1cm}]{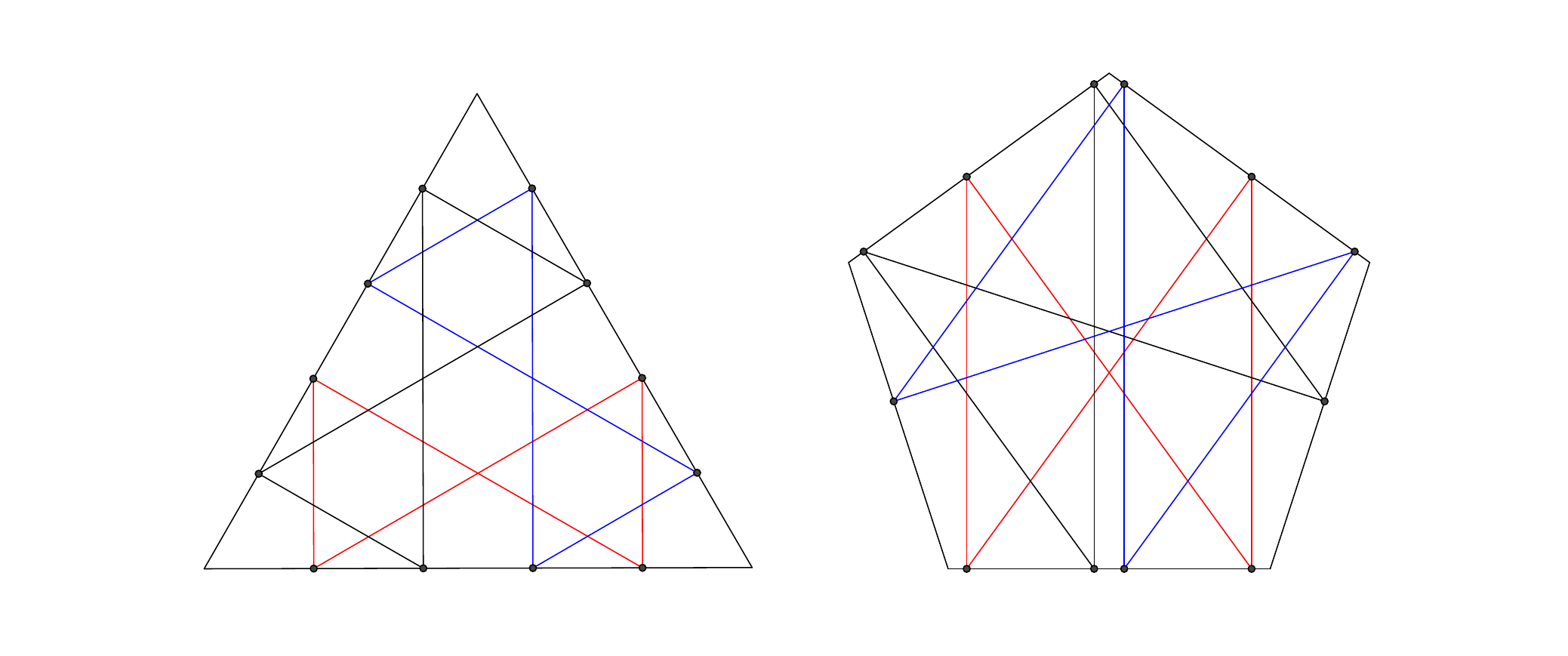}
    \caption{Illustrations for $O_{2k+1}(4) = 3$ for $k = 1$ (triangle) and $k = 2$ (pentagon).}
    \label{Fig-Prop4-3}
\end{figure}

\begin{Proposition}
$O_{2k+1}(2k+3) = 4k+2$ for all $k \geq 1$.
\label{prop-4}
\end{Proposition}

\begin{proof}
Again, write
\begin{align*}
    A_{2k+1}^{2k+3} &= \left(C_{2k+1}^k + C_{2k+1}^{k+1}\right)^{2k+3} \\
    &= \sum_{j=0}^{2t+1} \binom{2k+3}{j} C_{2k+1}^{kj} C_{2k+1}^{(k+1)(2k+3-j)}.
\end{align*}

The exponent simplifies to $$(k+2)(2k+1)+1-j,$$ which is a multiple of $2k+1$ only for $j = 1$ or $j = 2k+2$. Hence, the $(1,1)$ entry of $A_{2k+1}^{2k+3}$ is $$\dbinom{2k+3}{1}+\dbinom{2k+3}{2k+2} = 4k+6.$$

The only remaining partitions of $2k+3$ are $2+(2k+1)$ and $2k+3$ itself, since every partition must involve at least one odd number and we have seen in Proposition \ref{prop-2} that there are no pure sequences of sides with odd length smaller than $2k+1$. Hence, the periodic orbits in the partition $2+(2k+1)$ are being counted twice. Since $P_{2k+1}(2) = 2$ and $P_{2k+1}(2k+1) = 2$ (by Proposition \ref{prop-1}), then $$O_{2k+1}(2k+3) = 4k+6 - 2\cdot 2 = 4k+2.$$ An illustration for $O_3(5) = 6$ is given in Figure \ref{orbits-prop4-4}.

\end{proof}

\begin{figure}[h!]
    \centering
    \includegraphics[scale=0.7,trim={5.5cm 1cm 5.5cm 1cm}]{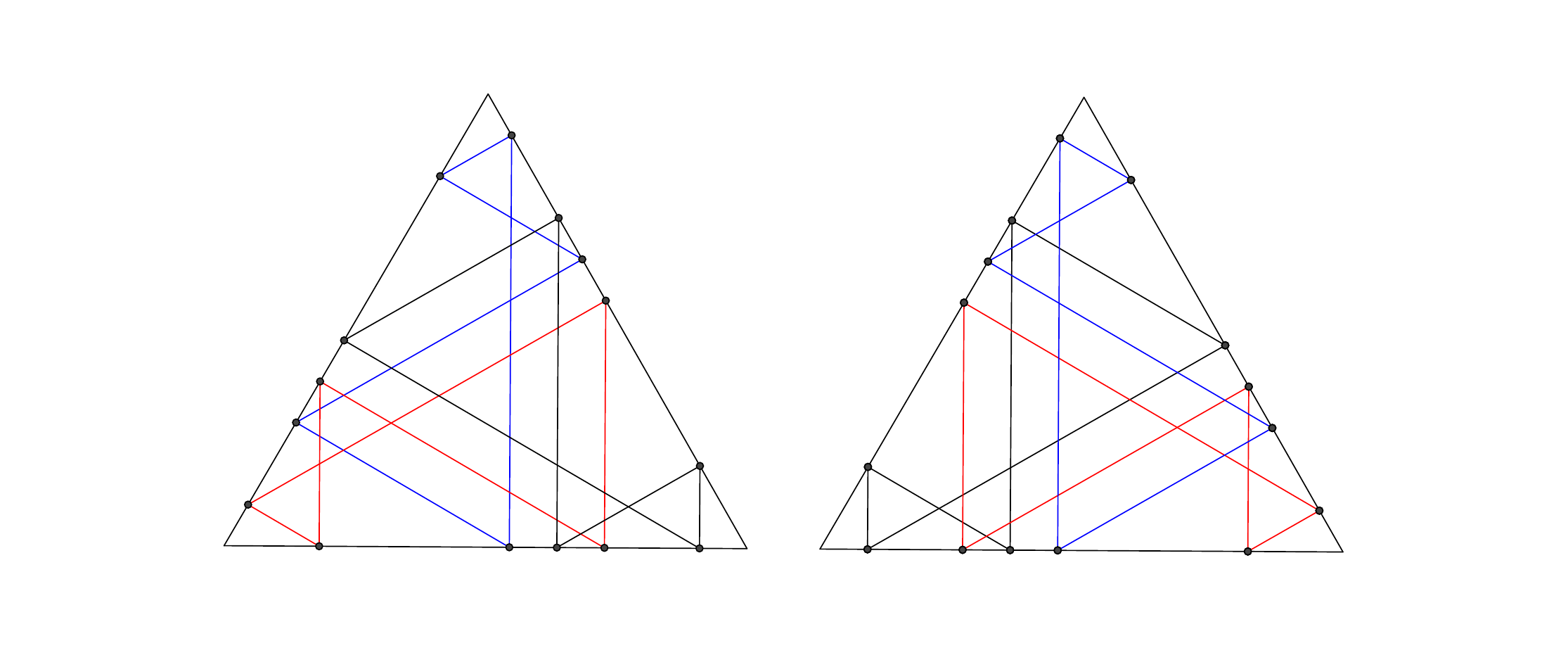}
    \caption{Illustrations for $O_{2k+1}(2k+3) = 4k+2$ for $k = 1$ (triangle).}
    \label{orbits-prop4-4}
\end{figure}

To finish this Section we note that all the periodic orbits appearing in Figures \ref{orbitas-prop4-1}, \ref{Fig-Prop4-3} and \ref{orbits-prop4-4} can be transformed into one another by a symmetry of the regular polygon. This does not happen every time, which can be seen, for example, in the case of the $6$-periodics originating from the sequences $(1,2,3,2,1,3)$ and $(1,2,1,3,1,2)$ in an equilateral triangle, shown in Figure \ref{6-periodics}.

\begin{figure}[h!]
    \centering
    \includegraphics[scale=0.7,trim={5.5cm 1cm 5.5cm 0cm}]{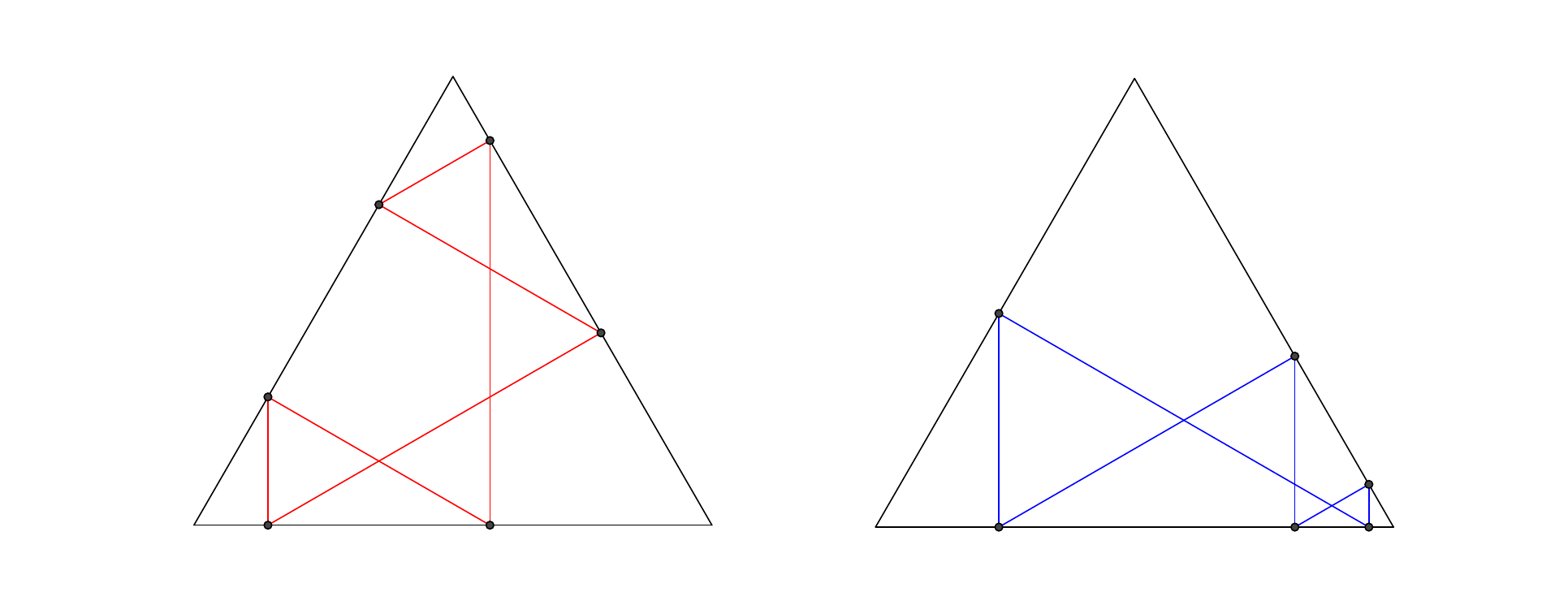}
    \caption{Two periodic orbits that cannot be transformed into each other by a symmetry of the regular polygon.}
    \label{6-periodics}
\end{figure}

\section{Applications}

\begin{Example}
$O_3(8) = 30$.
\end{Example}

\begin{proof}
It is easy to see that, in an equilateral triangle, there are exactly $2$ pure $r$-periodics for every $r \geq 2$, so $$\prod_{i=1}^m P_{2k+1}(p_i) = 2^m$$ in \eqref{formula-par1}. The partitions with at least two distinct parts are listed in Table \ref{example-1a}.

\begin{table}[!h]
 	\centering
 	\small
 	\caption{Partitions of $8$ with at least two distinct parts.}
 	\label{example-1a}
 	\begin{tabular}{cccc}
 		\hline 
 		$p$ & $m$ & $F(p)$ & $2^{m}F(p)$ \\ 
 		
 		 \hline
             $2+6$ & $2$ & $CP(1,1) = 1$ & $4$ \\
             $3+5$ & $2$ & $CP(1,1) = 1$ & $4$ \\
             $2+2+4$ & $3$ & $CP(2,1) = 1$ & $8$ \\
             $2+3+3$ & $3$ & $CP(1,2) = 1$ & $8$ \\
             Total   &     &               & $24$ \\
         		    \hline
 	\end{tabular}%
\end{table}

The trivial partition produces $2$ (pure) $8$-periodics. The other partitions with equal parts are $4+4$ and $2+2+2$ (with $n/d$ equal to $2$ and $4$, respectively). They are displayed in Table \ref{example-1b}.

\begin{table}[!h]
 	\centering
 	\small
 	\caption{The $4+4$ and $2+2+2+2$ partitions of $8$.}
 	\label{example-1b}
 	\begin{tabular}{cccccc}
 		\hline 
         $4+4$  & $(\alpha_1,\alpha_2)$ & $CP(\alpha_1,\alpha_2)$ & $2+2+2+2$  & $(\alpha_1,\alpha_2)$ & $CP(\alpha_1,\alpha_2)$ \\ 
 		
 		 \hline
 		    &   $(2,0)$ &  $1$ &   &  $(4,0)$ &  $1$ \\
 		    &   $(1,1)$ &  $1$ &   &  $(3,1)$ &  $1$ \\
 		    &   $(0,2)$ &  $1$ &   & $(2,2)$  &  $2$ \\
 		   Total    &   & $3$  &   &  $(1,3)$ &  $1$ \\
 		    &           &      &   &  $(0,4)$ &  $1$ \\
 		    &    & & Total   &   & $6$  \\
 		     \hline
 	\end{tabular}%
\end{table}

By \eqref{model1} and Proposition \ref{prop-3}, $$O_3(8) = 24 + 2 + 3 + 6 - O_3(2) - O_3(4) = 35 - 2 - 3 = 30.$$

\end{proof}

\begin{Example}
$O_5(6) = 10$.
\end{Example}

\begin{proof}
The partitions of $6$ are $2+2+2$, $2+4$, $3+3$ and $6$ itself. By Proposition \ref{prop-2} there are no pure $3$-periodics, so the $3+3$ partition does not need to be analyzed. As in Section 3, we consider the matrix $$B_5 = \left[ \begin{array}{ccccc} 0 & 0 & 0 & 0 & 0 \\ 0 & 0 & 0 & 1 & 1 \\ 0 & 0 & \tikzmark{left}0 & 0 & 1 \\ 0 & 1 & 0 & 0\tikzmark{right} & 0 \\ 0 & 1 & 1 & 0 & 0 \\ \end{array} \right].$$

The numbers of pure orbits with length $2$ is $2$, whereas for $4$ and $6$ they are given by the sum of entries in the following submatrices of $B_5^2$ and $B_5^4$, respectively:
$$B_5^2 = \left[ \begin{array}{ccccc} 0 & 0 & 0 & 0 & 0 \\ 0 & 2 & 1 & 0 & 0 \\ 0 & 1 & \tikzmark{left}1 & 0 & 0 \\ 0 & 0 & 0 & 1\tikzmark{right} & 1 \\ 0 & 0 & 0 & 1 & 2 \\ \end{array} \right]\DrawBox[thick], \quad B_5^4 = \left[ \begin{array}{ccccc} 0 & 0 & 0 & 0 & 0 \\ 0 & 5 & 3 & 0 & 0 \\ 0 & 3 & \tikzmark{left}2 & 0 & 0 \\ 0 & 0 & 0 & 2\tikzmark{right} & 3 \\ 0 & 0 & 0 & 3 & 5 \\ \end{array} \right]\DrawBox[thick].$$

Hence, $P_5(4) = 2$ and $P_5(6) = 4$. The partition $p = 2+4$ has $F(p) = CP(1,1) = 1$, so it produces $F(p)P_5(2)P_5(4) = 4$ periodic orbits.

Now let $p = 2+2+2$. There are two kinds of pure $2$ orbits and $n/2 = 3$, so there are $$CP(3,0) + CP(2,1) + CP(1,2) + CP(0,3) = 4$$ periodic orbits.

Finally, we conclude that $$O_5(6) = 4 + 4 + 4 - O_5(2) - O_5(3) = 10.$$

\end{proof}


\section*{Acknowledgement} The author would like to thank Liliana Gheorghe for the discussions that led to this work.
\bibliographystyle{acm}
\bibliography{bibliography.bib}
\end{document}